\documentclass[12pt]{amsart}

\usepackage[latin1]{inputenc}
\usepackage{amsfonts,amssymb,amsmath}
\usepackage{verbatim}
\input{xypic}
\usepackage[all,cmtip]{xy}
\newtheorem{theorem}{Theorem}[section]
\newtheorem{proposition}[theorem]{Proposition}
\newtheorem{lemma}[theorem]{Lemma}
\newtheorem{corollary}[theorem]{Corollary}
\newtheorem{definition}[theorem]{Definition}

\theoremstyle{remark}
\newtheorem{remark}[theorem]{Remark}

\renewcommand\leq{\leqslant}\renewcommand\geq{\geqslant}
\newcommand\Z{\ensuremath{\mathbb Z}}
\newcommand\Q{\ensuremath{\mathbb Q}}\newcommand\R{\ensuremath{\mathbb R}}
\newcommand\C{\ensuremath{\mathbb C}}
\newcommand\Qb{{\overline\Q}}\newcommand\Fb{{\bar F}}

\newcommand{\ra}{{\rightarrow}}
\newcommand{\lra}{\longrightarrow}

\newcommand\M{\operatorname{M}}

\newcommand\End{\operatorname{End}}

\newcommand\GL{\operatorname{GL}}

\newcommand\acc[2]{\ensuremath{{}^{#1}\hskip-0.1ex{#2}}}

\def\cA{\mathcal A}\def\cB{\mathcal B}\def\cD{\mathcal D}
\def\Fb{{\overline F}}\def\kb{{\bar k}}

\newcommand{\gp}{{\mathfrak{p}}}

\newcommand{\gP}{{\mathfrak{P}}}

\newcommand{\comp}{\begin{picture}(6,5)(-3,-2)\put(0,1){\circle{2}} \end{picture}}\def\circ{\comp}

\title[\Tiny{Abelian varieties with many endomorphisms and their simple factors}]{Abelian varieties with many endomorphisms  and their absolutely simple factors}

\author{Xavier Guitart}
\subjclass[2000]{11G10, 14K15.}
\address{ Departament de Matem\`atica Aplicada II,
 Universitat Polit\`ecnica de Ca\-ta\-lu\-nya,
Jordi Girona 1-3 (Edifici Omega) 08034, Barcelona}
\thanks{Partially supported by Grants MTM2009-13060-C02-01 and 2009 SGR 1220.}
\email{xevi.guitart@gmail.com}

\keywords{Abelian varieties of $\GL_2$-type, $k$-varieties,
building blocks.}

\date{\today}
\begin{document}
\begin{abstract}
We characterize the abelian varieties arising as absolutely simple factors
of $\GL_2$-type varieties over a number field $k$. In order to obtain this
result, we study a wider class of abelian varieties: the $k$-varieties $A/k$
satisfying that $\End_k^0(A)$ is a maximal subfield of $\End_\kb^0(A)$. We
call them \emph{Ribet-Pyle varieties} over $k$. We see that every Ribet-Pyle
variety over $k$ is isogenous over $\kb$ to a power of an abelian
$k$-variety and, conversely, that every abelian $k$-variety occurs as the
absolutely simple factor of some Ribet-Pyle variety over $k$. We deduce from
this correspondence a precise description of the absolutely simple factors
of the varieties over $k$ of $\GL_2$-type.
\end{abstract}

\maketitle

\section{Introduction}Let $k$ be a number field.  An abelian variety $A$ over  $k$ is said to be of $\GL_2$-type if
its algebra of $k$-endomorphisms $
\End_k^0(A)=\Q\otimes_\Z\End_k(A)$ is a number field of degree
equal to the dimension of $A$. The aim of this note is to
characterize the abelian varieties over $\kb$ that arise as
absolutely simple factors of  $\GL_2$-type varieties over $k$.

The interest in abelian varieties over $\Q$ of $\GL_2$-type arose in
connection with the Shimura-Taniyama conjecture on the modularity of
elliptic curves over $\Q$, and its generalization to higher dimensional
modular abelian varieties over $\Q$. To be more precise, to each  $A/\Q$ of
$\GL_2$-type is attached a compatible system of $\lambda$-adic
representations $\rho_{A,\lambda}\colon G_\Q\ra \GL_2(E_\lambda)$, where
$E=\End_\Q^0(A)$ and the $\lambda$'s are primes of $E$. As a consequence of
Serre's conjecture on Galois representations these $\rho_{A,\lambda}$ are
modular; that is, there exists a newform $f\in S_2(\Gamma_1(N))$ such that
$\rho_{A,\lambda}\simeq \rho_{f,\lambda}$ for all primes $\lambda$ of $E$,
where $\rho_{f,\lambda}$ is the $\lambda$-adic representation attached to
$f$ (see \cite{ribet-avQ} for the details).


The study of the $\Qb$-simple factors of $\GL_2$-type varieties over $\Q$
was initiated by K. Ribet in \cite{ribet-avQ}, in which the one-dimensional
factors where characterized: they are the elliptic curves $C/\Qb$ that are
isogenous to all their Galois conjugates, also known as elliptic
$\Q$-curves. This result was completed by Ribet's student E. Pyle in her PhD
thesis \cite{pyle}, where she characterized the higher dimensional
$\Qb$-simple factors  as a certain type of abelian $\Q$-varieties called
building blocks. More concretely, an abelian variety $B/\Qb$ is  an
\emph{abelian $\Q$-variety} if it is $\End_\Qb(B)$-equivariantly isogenous
to all of its Galois conjugates; this means that  for each $\sigma\in G_\Q$
there exists an isogeny $\mu_\sigma\colon \acc\sigma B\ra B$ such that
$\varphi\circ\mu_\sigma=\mu_\sigma \circ\acc\sigma\varphi$ for all
$\varphi\in \End_\Qb(B)$. A \emph{building  block} is an abelian
$\Q$-variety $B$ whose endomorphism algebra is a central division algebra
over a totally real field $F$, with Schur index $t\leq 2$ and reduced degree
$t[F:\Q]=\dim B$. The following statement is Proposition 1.3 and Proposition
4.5 of \cite{pyle}.
\begin{theorem}[Ribet-Pyle]\label{theorem: Ribet-Pyle}
Let $A/\Q$ be an abelian variety of $\GL_2$-type such that $A_\Qb$ does not
have complex multiplication. Then $A_\Qb$ decomposes up to $\Qb$-isogeny as
$A_\Qb\sim B^n$ for some building block $B/\Qb$. Conversely, if $B/\Qb$ is a
building block then there exists a $\GL_2$-type variety $A/\Q$ such that
$A_\Qb\sim B^n$ for some $n$.
\end{theorem}
Observe that this result establishes a correspondence between
abelian varieties of $\GL_2$-type over $\Q$ without CM and
building blocks. In the last chapter of Pyle's thesis, a series of
questions were posed about whether a similar correspondence holds
for $\GL_2$-type varieties over other fields $k$. The  goal of
this note is to establish such correspondence when $k$ is a number
field. In this case, the analogous of building blocks are abelian
$k$-varieties (that is, varieties $B/\kb$ equivariantly isogenous
to $\acc\sigma B$ for all $\sigma\in G_k$) whose endomorphism
algebra is a central division algebra over a field $F$ with Schur
index $t\leq 2$ and $t[F:\Q]=\dim B$. We call these varieties
\emph{building $k$-blocks}. We prove in Section \ref{section:
Varieties over $k$ of GL_2-type and k-varieties} that every
$\GL_2$-type variety $A/k$ such that $A_\kb$ does not have CM is
$\kb$-isogenous to the power of a building $k$-block. Conversely,
every building $k$-block arises as the $\kb$-simple factor of some
variety over $k$ of $\GL_2$-type. In other words, we construct a
correspondence
\begin{equation}\label{equation: correspondence between GL_2-type varieties and building k-blocks}
\frac{\{\text{$A/k$ of $\GL_2$-type without CM}
\}}{\text{$k$-isogeny}}\longleftrightarrow \frac{\{\text{building
$k$-blocks $B/\kb$}\}}{\text{$\kb$-isogeny}}.\end{equation} This
can be seen as a natural generalization of the results of Ribet
and Pyle to a wider class of abelian varieties. Moreover, it is
worth noting that varieties over $k$ of $\GL_2$-type play a
similar role as their counterparts over $\Q$ with respect to
modularity: they are conjectured to be modular, at least when $k$
is totally real, in a similar sense as they are known to be
modular for $k=\Q$. Indeed, if $A/k$ is of $\GL_2$-type and $k$ is
a totally real number field, a generalization of the
Shimura-Taniyama conjecture predicts the existence of a Hilbert
modular form $f$ such that $\rho_{A,\lambda}\simeq
\rho_{f,\lambda}$ for all primes $\lambda$ of $E=\End_k^0(A)$. See
\cite[Conjecture 2.4]{darmon} for a precise statement.

Observe that in correspondence \eqref{equation: correspondence between
GL_2-type varieties and building k-blocks} the objects in the right hand
side are $k$-varieties whose endomorphism algebra satisfies certain
conditions. Instead of proving \eqref{equation: correspondence between
GL_2-type varieties and building k-blocks} directly, what we do is to
construct as a previous step a more general correspondence, in which the
right hand side is enlarged to all abelian $k$-varieties. As we will see,
the varieties that correspond to them in the left hand side  are then
varieties $A/k$ characterized by the fact that $A_\kb$ is a $k$-variety and
$\End_k^0(A)$ is a maximal subfield of $\End_\kb^0(A)$. We call the
varieties satisfying these properties \emph{Ribet-Pyle varieties}, because
they arise naturally in this generalization of the results of Ribet and
Pyle. Section \ref{section: Ribet-Pyle varieties} is devoted to the study of
Ribet-Pyle varieties and their absolutely simple factors, and we obtain the
following main result.
\begin{theorem}\label{theorem: main theorem}
Let $k$ be a number field and let $A/k$ be a Ribet-Pyle variety. Then
$A_\kb$ decomposes up to $\kb$-isogeny as $A_\kb\sim B^n$ for some abelian
$k$-variety $B/\kb$. Conversely, if $B/\kb$ is a $k$-variety then there
exists a Ribet-Pyle variety $A/k$ such that $A_\kb\sim B^n$ for some $n$.
\end{theorem}
This result  gives some insight into the nature of the correspondences of
Theorem \ref{theorem: Ribet-Pyle} and its generalization \eqref{equation:
correspondence between GL_2-type varieties and building k-blocks}. Indeed,
what we do  in Section \ref{section: Varieties over $k$ of GL_2-type and
k-varieties} is to  prove that varieties over $k$ of $\GL_2$-type without CM
are Ribet-Pyle varieties, and then we obtain \eqref{equation: correspondence
between GL_2-type varieties and building k-blocks} by applying  Theorem
\ref{theorem: main theorem} to $\GL_2$-type varieties.

\section{Ribet-Pyle varieties}\label{section: Ribet-Pyle varieties}
Let $k$ be a number field. In this section we establish and prove the
correspondence between abelian $k$-varieties and  Ribet-Pyle varieties of
Theorem \ref{theorem: main theorem}. We begin by giving the relevant
definitions.
\begin{definition}\label{definition: abelian k-varieties}
An abelian variety $B/\kb$ is  an \emph{abelian $k$-variety} if for each
$\sigma \in G_k$ there exists an isogeny $\mu_\sigma\colon \acc\sigma B\ra
B$ compatible with the endomorphisms of $B$; i.e., such that for all
$\varphi\in \End_\kb(B)$ the following diagram is commutative
\begin{equation}\label{equation: diagram of compatibility of isogenies}
 \xymatrix{ {^\sigma B} \ar[d]_{^\sigma\varphi} \ar[r]^{\mu_\sigma  }
&{B}\ar[d]_{\varphi}\\
{^\sigma B}  \ar[r]^{\mu_\sigma } &{ B} .}
\end{equation}
\end{definition}
\begin{definition}
An abelian variety $A$ defined over $k$ is a \emph{Ribet-Pyle variety}  if
$A_\kb$ is an abelian $k$-variety and $\End_k^0(A)$ is a maximal subfield of
$\End_\kb^0(A)$.
\end{definition}
\begin{remark}
We remark that not all abelian varieties $A$ defined over $k$
satisfy that $A_\kb$ is a $k$-variety. Indeed, although in this
case the identity is an obvious isogeny between $\acc\sigma A $
and $A$, it is not necessarily compatible with $\End_{\kb}(A)$ in
general.
\end{remark}
One of the directions of the correspondence that we aim to establish follows
almost immediately from the definitions.

\begin{proposition}\label{prop: Ribet-Pyle varieties are isotypical}
Let $A/k$ be a Ribet-Pyle variety. Then it decomposes up to $\kb$-isogeny as
$A_\kb\sim B^n$, for some  simple abelian $k$-variety $B$ and some $n$.
\end{proposition}
\begin{proof}
Let $F$ be the center of $\End_\kb^0(A)$ and let $\varphi$ be an element of
$F$. Since $A_\kb$ is a $k$-variety, for each $\sigma \in G_k$ we have that
\begin{equation}\label{equation: compatible isogenies for a Ribet-Pyle
A}\acc\sigma
\varphi=\mu_\sigma^{-1}\circ\varphi\circ\mu_\sigma,\end{equation} for some
isogeny
 $\mu_\sigma\colon \acc\sigma A_\kb\ra A_\kb$. Since $A$ is defined
over $k$ the isogeny $\mu_\sigma$ belongs  to $\End_\kb^0(A)$. Then
$\acc\sigma\varphi=\varphi$ because $\varphi$ belongs to the center of
$\End_\kb^0(A)$. This gives the inclusion $F\subseteq \End_k^0(A)$. By
hypothesis $\End_k^0(A)$ is a field, so $F$ is a field as well and this
implies that $A_\kb\sim  B^n$ for some simple variety $B$ and some $n$.
Next, we show that $B$ is a $k$-variety. By fixing an isogeny $A_\kb\sim
B^n$ the center of $\End_\kb^0(B)$ can be identified with $F$, and each
compatible isogeny $\mu_\sigma\colon \acc\sigma A_\kb\ra A_\kb$ gives rise
to an isogeny $\nu_\sigma\colon \acc\sigma B\ra B$. The relation
\eqref{equation: compatible isogenies for a Ribet-Pyle A} implies that $
\psi=\nu_\sigma\circ\acc\sigma\psi\circ\nu_\sigma^{-1}$ for all $\psi\in
Z(\End_\kb^0(B))\simeq F$, so that  the map
$$\begin{array}{ccc}\End_\kb^0(B)&\lra &\End_\kb^0(B)\\
\psi &\longmapsto & \nu_\sigma\circ\acc\sigma\psi\circ\nu_\sigma^{-1}
\end{array}$$
is a $F$-algebra automorphism. By the Skolem-Noether Theorem it is inner,
and there exists an element $\alpha_\sigma\in \End_\kb^0(B)^*$ such that
$$\nu_\sigma\circ\acc\sigma\psi\circ\nu_\sigma^{-1}=\alpha_\sigma^{-1}\circ\psi\circ\alpha_\sigma,$$
for all $\psi\in \End_\kb^0(B)$. The isogeny $\alpha_\sigma\circ\nu_\sigma$
satisfies the compatibility condition \eqref{equation: diagram of
compatibility of isogenies} and we see that $B$ is a $k$-variety.

\end{proof}
The  following statement gives the other direction of the correspondence
between $k$-varieties and Ribet-Pyle varieties in the number field case.
\begin{theorem}\label{theorem: Ribet-Pyle variety associated to a k-variety}
Let $k$ be a number field, and let $B/\kb$ be a simple abelian $k$-variety.
Then there exists a Ribet-Pyle variety $A/k$ such that $A_\kb\sim B^n$ for
some $n$.
\end{theorem}
Before giving the proof of Theorem \ref{theorem: Ribet-Pyle variety
associated to a k-variety} we shall need some preliminary results.

\subsection*{Cohomology classes and splitting fields} Let $k$ be a number field and let $B/\kb$ be a simple
abelian $k$-variety. Let $\cB$ be its endomorphism algebra and let
$F$ be the center of $\cB$. Since $B$ has a model over a finite
extension of $k$, we can choose for each $\sigma\in G_k$ a
compatible isogeny $\mu_\sigma: \acc\sigma B\ra B$ in such a way
that the set $\{\mu_\sigma\}_{\sigma\in G_k}$ is locally constant;
more precisely, such that $\mu_\sigma=\mu_\tau$ if $\acc\sigma
B=\acc \tau B$. Then we can define a map $c_B:G_k\times G_k\ra
F^*$ by means of
$c_B(\sigma,\tau)=\mu_\sigma\circ\acc\sigma\mu_\tau\circ\mu_{\sigma\tau}^{-1}$.
It is easy to check that $c_B$ is a continuous 2-cocycle of $G_k$
with values in $F^*$ (considering the trivial action of $G_k$ in
$F^*$). Its cohomology class $[c_B]\in H^2(G_k,F^*)$ is an
invariant of the isogeny class of $B$ and it is independent of the
compatible isogenies used to define it.

The inclusion of $G_k$-modules with trivial action $F^*\hookrightarrow
\Fb^*$ induces a homomorphism between the cohomology groups $H^2(G_k,F^*)\ra
H^2(G_k,\Fb^*)$. A theorem of Tate implies that $H^2(G_k,\Fb^*)=\{1\}$
(see~\cite[Theorem 6.3 ]{ribet-avQ}). Therefore, the image of $[c_B]$ in
$H^2(G_k,\Fb^*)$ is trivial, which means that there exist continuous maps
$\beta:G_k\ra \Fb^*$ such that
\begin{equation}\label{equation: splitting map def}c_B(\sigma,\tau)=\beta(\sigma)\beta(\tau)\beta(\sigma\tau)^{-1}.\end{equation}
We say that a map $\beta$ satisfying \eqref{equation: splitting map def} is
a \emph{splitting map} for the cocycle $c_B$. If $\chi\colon G_k\ra \Fb^*$
is a character then $\beta'=\beta\chi$ is another splitting map for $c_B$.
In fact, as we vary $\chi$ through all the characters from $G_k$ to $\Fb^*$
we obtain all the splitting maps for $c_B$.  For a splitting map $\beta$, we
will denote by $E_\beta$ the field $F(\{\beta(\sigma)\}_{\sigma\in
G_k})\subseteq \Fb$. The extension $E_\beta/F$ is finite because  $\beta $
is continuous.

 Let $m$ be the order of $[c_B]$ in $H^2(G_k,F^*)$, and let
$d$ be a continuous map $d\colon G_k\ra F^*$ expressing $c_B^m$ as a
coboundary:
\begin{equation}\label{equation: m-power of c_B as a coboundary}
c_B(\sigma,\tau)^m=d(\sigma)d(\tau)d(\sigma\tau)^{-1}.
\end{equation}We define a map
\begin{equation*}
\begin{array}{cccc}
\varepsilon_\beta\colon  &G_k&\lra & \Fb^*\\
 & \sigma & \longmapsto &\beta(\sigma)^m/d(\sigma).
\end{array}
\end{equation*}
By \eqref{equation: splitting map def} and \eqref{equation: m-power of c_B as a coboundary} we see that $\varepsilon_\beta\colon G_k\ra \Fb^*$
is a continuous character.
\begin{lemma}\label{lemma: existence of splitting fields containing cyclotomic extensions}
For each nonnegative integer $n$ there exists a splitting map
$\beta$ such that $F(\zeta_n)\subseteq E_\beta$, where $\zeta_n$ is
a primitive $n$-th root of unity in $\Fb$.
\end{lemma}
\begin{proof}
Let $\beta'$ be a splitting map for $c_B$, and let $r$ be the
order of $\varepsilon_{\beta'}$. Let $e=\gcd(n,r)$ and let
$\chi\colon G_k\ra \Fb^*$ be a character of order $mn/e$,
where $m$ is the order of $[c_B]$ in $H^2(G_k,F^*)$. Then the
character $\chi^m\varepsilon_{\beta'}$ is the character that
corresponds to the splitting map $\beta=\chi\beta'$ and its order
is $nr/e$, which is a multiple of $n$. Therefore $E_\beta$
contains a primitive $n$-th root of unity $\zeta_n$.
\end{proof}

\subsection*{  Cyclic splitting fields of simple algebras} Let $\cA$
be a central simple algebra over a number field $F$. A well-known result of
central simple algebras over number fields  guarantees the existence of
fields $L$ cyclic over $F$ that split $\cA$ (i.e. with $\cA\otimes_FL\simeq
\M_n(L)$ for some $n$). In order to prove Theorem~\ref{theorem: Ribet-Pyle
variety associated to a k-variety} we  use a similar result, but with the
extension $L$ being cyclic over $\Q$ and such that  $LF$ splits $\cA$.
Although this is probably also well-known, for lack of reference we  include
a proof based on   the Grunwald-Wang Theorem.

\begin{theorem}[Grunwald-Wang Theorem]
Let $M$ be a number field, and let $\{(v_1,n_1),\ldots,(v_r,n_r)\}$ be a
finite set of pairs, where each $v_i$ is a place of $M$ and each $n_i$ is a
positive integer such that $n_i\leq 2$ if $v_i$ is a real place, and $n_i=1$
if $v_i$ is a complex place. Let $m$ be the least common multiple of the
$n_i$'s, and let $n$ be a positive integer divisible by $m$. Then there
exists a cyclic extension $L/M$ of degree $n$ such that for each $i$ the
degree $[L_{v_i}:M_{v_i}]$ is divisible by $n_i$.
\end{theorem}

\begin{proposition}
Let $F$ be a number field and let $\cD$ be a central division algebra over
$F$. There exists a cyclic extension $L/\Q$ such that $LF$ is a splitting
field for  $\cD$.
\end{proposition}
\begin{proof}
Let $F'$ be the Galois closure of $F$. Let $n=[F':\Q]$ and let $t$
be the Schur index of $\cD$. Let $\{\gp_1,\ldots,\gp_s\}$ be the
set of primes of $F$ where $\cD$ ramifies, and let
$\{p_1,\ldots,p_l\}$ be the set of primes of $\Q$ below
$\{\gp_1,\ldots,\gp_s\}$.  The Grunwald-Wang Theorem, when applied
to the primes $p_i$ with $n_i=tn$, and to the infinite place of
$\Q$ with $n_\infty=2$, guarantees the existence of a cyclic
extension $L/\Q$ of degree $2tn$ such that $[L_p:\Q_p]=tn$ for all
$p$ belonging to $\{p_1,\ldots,p_l\}$ and $L_v=\C$ for all
archimedean places $v$ of $L$. Let $K=LF$.

If $\gp$ is a prime of $F$  dividing $p$, and $\gP$ is  a prime of $K$
dividing $\gp$, the fields $L_p$ and $F_\gp$ can be seen as subfields of
$K_\gP$. Then the degree $g=[L_p\cap F_\gp:\Q_p]$ divides $n$, so
$[L_p\colon L_p\cap F_\gp]=t\frac{n}{g}=[F_\gp L_p\colon F_\gp]$ and we see
that $t$ divides $[K_\gP:F_\gp]$. Therefore, $K$ is a totally imaginary
extension of $F$ such that, for every prime $\gp$ of $F$ ramifying in $\cD$
and for every prime $\gP$ of $K$ dividing $\gp$, the index $[K_\gP:F_\gp]$
is a multiple of the Schur index of $\cD$. This implies that $K$ is  a
splitting field for $\cD$ (see \cite[Corollary 18.4 b and Corollary 17.10
a]{Pi}).
\end{proof}
\begin{corollary}\label{corollary: existence of splitting fields containing roots of unity}
Every central division $F$-algebra is split by an extension of the form
$F(\zeta_m)$ for some $m$.
\end{corollary}
\begin{proof}
By the previous proposition there exists a cyclic extension $L/\Q$ such that
$LF$ splits $\cD$. The field $L$ is contained in a field of the form
$\Q(\zeta_m)$ by the Kronecker-Weber Theorem, and then $F(\zeta_m)$ splits
$\cD$.
\end{proof}
\subsection*{Construction of Ribet-Pyle varieties} In this paragraph we
perform the construction of Ribet-Pyle varieties having a $k$-variety $B$ as
simple factor.  Recall that $\cB$ denotes $\End_\kb^0(B)$,
 $F$ is the center of $\cB$ and $t$ denotes the Schur index of $\cB$.
 Fix also   a  locally constant set of
 isogenies $\{\mu_\sigma\colon  \acc\sigma B\ra B\}_{\sigma\in G_k}$, let $c_B$ be the cocycle constructed with these isogenies and let
 $\beta$ be a splitting map for $c_B$.

Let $n$ be the degree $[E_\beta:F]$, and fix an injective $F$-algebra
homomorphism
\begin{equation*}
\phi\colon E_\beta\lra \M_n(F)\subseteq \M_n(\cB)\simeq \End_\kb^0(B^n).
\end{equation*}
The elements of $E_\beta$ act as endomorphisms of $B^n$ up to isogeny by
means of $\phi$. Let $\hat\mu_\sigma$ be the diagonal isogeny
$\hat\mu_\sigma\colon \acc\sigma B^n\ra B^n$ consisting in $\mu_\sigma$ in
each factor.
\begin{proposition}\label{prop: definition of X_beta}
There exists an abelian variety $X_\beta$ over $k$ and a $\kb$-isogeny
$\kappa\colon B^n\ra X_\beta$  such that
$\kappa^{-1}\circ\acc\sigma\kappa=\phi(\beta(\sigma))^{-1}\circ \hat
\mu_\sigma$ for all $\sigma\in G_k$. Moreover, the $k$-isogeny class of
$X_\beta$ is independent of the chosen injection $\phi$.
\end{proposition}
\begin{proof}
Let $\nu_\sigma$ be the  isogeny defined as
$\nu_\sigma=\phi(\beta(\sigma))^{-1}\circ\hat\mu_\sigma$. In order to prove
the existence of $X_\beta$, by \cite[Theorem 8.1]{ribet-avQ} we need to
check that $\nu_\sigma\circ\acc\sigma\nu_\tau\circ\nu_{\sigma\tau}^{-1}=1$.
By the compatibility of $\mu_\sigma$ we have that:
\begin{eqnarray*}
\nu_\sigma\circ{^\sigma
\nu_\tau}\circ\nu_{\sigma\tau}^{-1}&=&\phi(\beta(\sigma))^{-1}\circ\hat{\mu}_\sigma\circ
{^\sigma \phi(\beta(\tau))^{-1}}\circ{^\sigma \hat{\mu}_\tau}\circ
\hat{\mu}_{\sigma\tau}^{-1}\circ \phi(\beta(\sigma\tau))\\
&=&\phi(\beta(\sigma))^{-1}\circ\phi(\beta(\tau))^{-1}\circ\hat{\mu}_\sigma\circ
{^\sigma \hat{\mu}_\tau}\circ \hat{\mu}_{\sigma\tau}^{-1} \circ\phi(\beta(\sigma\tau))\\
&=&\phi(\beta(\sigma))^{-1}\circ\phi(\beta(\tau))^{-1}\circ c_B(\sigma,\tau) \circ\phi(\beta(\sigma\tau))\\
&=&\phi(\beta(\sigma)^{-1}\circ\beta(\tau)^{-1} \circ\beta(\sigma\tau))\circ c_B(\sigma,\tau)\\
&=&\phi(c_B(\sigma,\tau)^{-1})\circ c_B(\sigma,\tau)=c_B(\sigma,\tau)^{-1}\circ c_B(\sigma,\tau)=1.\\
\end{eqnarray*}
Now suppose that $\phi$ and $\psi$ are $F$-algebra homomorphisms $E_\beta\ra
\M_n(F)$, and let $X_{\beta,\phi}$ and $X_{\beta,\psi}$ denote the varieties
constructed by the above procedure using $\phi$ and $\psi$ respectively to
define the action of $E_\beta$ on $B^n$. We aim to see that $X_{\beta,\phi}$
and $X_{\beta,\psi}$ are $k$-isogenous.

Let $C$ denote the image of $\phi$. The map $\phi(x)\mapsto\psi(x)\colon
C\ra \M_n(F)$ is a $F$-algebra homomorphism. Since $C$ is simple and
$\M_n(F)$ is central simple over $F$, by the Skolem-Noether Theorem there
exists an element $b$ in $\M_n(F)$ such that $\phi(x)=b\psi(x)b^{-1}$ for
all $x$ in $E_\beta$. By the defining property of $X_{\beta,\phi}$ and
$X_{\beta,\psi}$ there exist $\kb$-isogenies $\kappa\colon B^n\lra
X_{\beta,\phi}$ and $\lambda\colon B^n\lra X_{\beta,\psi}$ such that
\begin{equation}\label{eq:11} \kappa^{-1}\circ
{^\sigma\kappa}=\phi(\beta(\sigma))^{-1}\circ
\hat{\mu}_\sigma=b\circ\psi(\beta(\sigma))^{-1}\circ
b^{-1}\circ\hat{\mu}_\sigma,\end{equation}
\begin{equation}\label{eq:12}
 \lambda^{-1}\circ
{^\sigma\lambda}=\psi(\beta(\sigma))^{-1}\circ \hat{\mu}_\sigma.
\end{equation}
The $\kb$-isogeny $\nu=\kappa\circ b\circ \lambda^{-1}\colon
X_{\beta,\psi}\lra X_{\beta,\phi}$ is in fact defined over $k$, since for
each $\sigma$ of $G_k$ we have that
\begin{eqnarray*}
\nu^{-1}\circ{^\sigma\nu}&=&\lambda\circ b^{-1}\circ\kappa^{-1}\circ{^\sigma
\kappa}\circ{^\sigma
b}\circ{^\sigma \lambda^{-1}}\\
&=&\lambda\circ b^{-1}\circ b\circ \psi(\beta(\sigma))^{-1}\circ
b^{-1}\circ\hat{\mu}_\sigma \circ{^\sigma b}\circ{^\sigma
\lambda^{-1}}\\
&=&\lambda\circ\psi(\beta(\sigma))^{-1}\circ\hat{\mu}_\sigma \circ{^\sigma
b^{-1}}\circ{^\sigma
b}\circ{^\sigma\lambda^{-1}}\\
&=&\lambda\circ \lambda^{-1}\circ{^\sigma \lambda}\circ{^\sigma \lambda
^{-1}}=1,
\end{eqnarray*}
where we used the compatibility of $\hat\mu_\sigma$ with the endomorphisms
of $B^n$ in the third equality, and the expressions \eqref{eq:11} and
\eqref{eq:12} in the second and fourth equality respectively.
\end{proof}
\begin{proposition}\label{prop: k-rational endomorphisms of X_beta}
The algebra $\End_k^0(X_\beta)$ is isomorphic to the centralizer of
$E_\beta$ in $\M_n(\cB)$.
\end{proposition}
\begin{proof}
 $\End_\kb^0(X_\beta)$ is isomorphic to $\M_n(\cB)$ and  every
endomorphism of $X_\beta$ up to $\kb$-isogeny is of the form $\kappa\circ
\psi\circ \kappa^{-1}$, for some $\psi\in \End_\kb^0(B^n)$. For $\sigma$ in
$G_k$ we have:
\begin{eqnarray*}
{^\sigma(\kappa \circ \psi\circ \kappa^{-1})}=\kappa\circ \psi\circ
\kappa^{-1} &\iff& {^\sigma\kappa}\circ {^\sigma\psi}\circ
{^\sigma\kappa^{-1}}=\kappa \circ\psi \circ\kappa^{-1} \\&\iff&
\kappa^{-1}\circ{^\sigma\kappa}\circ{^\sigma \psi}
\circ(\kappa^{-1}\circ{^\sigma\kappa})^{-1}=\psi \\ &\iff &
\beta(\sigma)\circ{\hat{\mu}_\sigma}\circ{^\sigma\psi}\circ
{\hat{\mu}_\sigma}^{-1}\circ\beta(\sigma)^{-1}=\psi \\&\iff& \beta(\sigma)\circ\psi\circ\beta(\sigma)^{-1}=\psi.\\
\end{eqnarray*}
Thus the endomorphisms of $X_\beta$ defined over $k$ are exactly the ones
coming from endomorphisms $\psi$ that commute with $\beta(\sigma)$, for all
$\sigma$ in $ G_k$. Now the proposition is clear, since the
$\beta(\sigma)$'s generate $E_\beta$.
\end{proof}
\begin{corollary}\label{corollary: k-rational endomorphisms of X_beta}
The algebra $\End_k^0(X_\beta)$ is isomorphic to $E_\beta\otimes_F \cB $.
\end{corollary}
\begin{proof}
Let $C$ be the centralizer of $E_\beta$ in $\M_n(\cB)$. In view of
Proposition \ref{prop: k-rational endomorphisms of X_beta} we have to prove
that $C\simeq E_\beta\otimes_F \cB$. It is clear that $E_\beta$ is contained
in $C$. Moreover, $\cB$ is contained in $C$ because the elements of
$E_\beta$ can be seen as $n\times n$ matrices with entries in $F$, and these
matrices commute with $\cB$ (which is identified with the diagonal matrices
in $\M_n(\cB)$). Since $E_\beta$ and $\cB$ commute there exists a subalgebra
of $C$ isomorphic to $E_\beta\otimes_F \cB$, which has dimension $nt^2$ over
$F$. By the Double Centralizer Theorem we know that
\[
[C:F][E_\beta:F]=[\M_n(\cB):F]=n^2t^2,
\]
and from this we obtain  that $[C:F]=nt^2$, hence $C$ is isomorphic to
$E_\beta\otimes_F \cB$.
\end{proof}
At this point we have at our disposal all the tools needed to prove Theorem
\ref{theorem: Ribet-Pyle variety associated to a k-variety}.
\begin{proof}[Proof of Theorem \ref{theorem: Ribet-Pyle variety associated to a k-variety}]
By Corollary \ref{corollary: existence of splitting fields containing roots
of unity} there exists an integer $m$ such that $F(\zeta_m)$ splits $\cB$.
Let $\beta$ be a splitting map for $c_B$ with $E_\beta$ containing
$F(\zeta_m)$; the existence of such a $\beta$ is guaranteed by Lemma
\ref{lemma: existence of splitting fields containing cyclotomic extensions}.
Consider the variety $X_\beta$ defined as in Proposition \ref{prop:
definition of X_beta}. By Corollary \ref{corollary: k-rational endomorphisms
of X_beta} we have that $\End_k^0(X_\beta)\simeq E_\beta\otimes_F \cB$, and
this later algebra is in turn isomorphic to $M_t(E_\beta)$ because $E_\beta$
is a splitting field for $\cB$. Therefore, there exists an abelian variety
$A_\beta$ defined over $k$ such that $X_\beta\sim_k A_\beta^t$ and
$\End_k^0(A_\beta)\simeq E_\beta$. Clearly $A_\beta$ is $\kb$-isogenous  to
$B^{n/t}$, where $n=[E_\beta:F]$, and we claim that it is  a Ribet-Pyle
variety. First of all, it is easily seen that the power of a $k$-variety is
also a $k$-variety. This implies that $(A_\beta)_\kb$ is a $k$-variety.
Moreover, we have that $[ \End_k^0(A_\beta)\colon F]=[E_\beta\colon F]=n$,
and the dimension of the ambient algebra is
 $[
\End_\kb^0(A_\beta)\colon F]=(\frac{n}{t})^2[\cB\colon F]=n^2$. This implies
(cf. \cite[Proposition 13.1]{Pi}) that $\End_k^0(A)$ is  a maximal subfield
of $\End_\kb^0(A)$.
\end{proof}
\begin{proposition}Let $B$ be a $k$-variety and let $A/k$ be a Ribet-Pyle
variety having $B$ as $\kb$-simple factor. Then $A$ is $k$-isogenous to the
variety $A_\beta$ obtained by applying the above procedure to some cocycle
$c_B$ attached to $B$ and some  splitting map $\beta$ for $c_B$.
\end{proposition}
\begin{proof}
 Let $\cB=\End_\kb^0(B)$, let $F$ be the center of $\cB$ and let $t$ be the Schur
index of $\cB$. Let $E$ be the maximal subfield $\End_k^0(A)$ of
$\End_\kb^0(A)$, and fix an embedding  of $E$ into $\Fb$. Let $\kappa$ be an
isogeny $\kappa\colon B^n\ra A_\kb$.  We have the relation $[E:F]=nt$. Let
$\{\mu_\sigma\colon \acc\sigma B\ra B\}_{\sigma\in G_k}$ be a locally
constant set of compatible isogenies and denote by $\hat \mu_\sigma\colon
\acc\sigma B^n\ra B^n$ the diagonal of $\mu_\sigma$. Define
$\beta(\sigma)=\kappa\circ \hat\mu_\sigma\circ\acc\sigma\kappa^{-1}$, which
is a compatible isogeny $\beta(\sigma)\colon A_\kb\ra A_\kb$. The fact that
$\beta(\sigma)$ is compatible implies that
\begin{equation}\label{eq: C(Ebeta) inside E}\beta(\sigma)\circ\varphi=\acc\sigma\varphi\circ\beta(\sigma)\end{equation} for all
$\sigma$ in $G_k$ and for all $\varphi \in \End_\kb^0(A)$. In particular,
when applied to elements $\varphi$ of $E$ this property says that
$\beta(\sigma)$ lies in $C(E)$, the centralizer of $E$. But $C(E)$ is equal
to $E$, because $E$ is a maximal subfield. Thus $\beta(\sigma)$ belongs to
$E$ and it is an isogeny defined over $k$. Now we have that
\begin{eqnarray*}c_B(\sigma,\tau)&=&\mu_\sigma\circ\acc\sigma\mu_\tau\circ\mu_{\sigma\tau}^{-1}=
\hat{\mu}_\sigma\circ\acc\sigma{\hat{\mu}_\tau}\circ\hat{\mu}_{\sigma\tau}^{-1}\\&=&\beta(\sigma)\circ\acc\sigma\beta(\tau)\circ\beta(\sigma\tau)^{-1}=
\beta(\sigma)\circ\beta(\tau)\circ\beta(\sigma\tau)^{-1},\end{eqnarray*} and
we see that the map $\sigma\mapsto \beta(\sigma)$ is  a splitting map for
$c_B$. We have already seen the inclusion $E_\beta\subseteq E$.  From
\eqref{eq: C(Ebeta) inside E} it is clear that $C(E_\beta)\subseteq E$, and
taking centralizers and applying the Double Centralizer Theorem we have that
$E=C(E)\subseteq C(C(E_\beta))=E_\beta$. Thus $E=E_\beta$ and, in
particular, $[E_\beta:F]=nt$.

Now we define a $\kb$-isogeny $\hat\kappa\colon  (B^n)^t\ra A_\kb^t$ as the
diagonal isogeny associated to $\kappa$, and we  make $E_\beta$ act on
$B^{nt}$ by means of $\hat\kappa$. It is easy to check that
$\hat\kappa^{-1}\circ\acc\sigma{\hat\kappa}=\hat\kappa^{-1}\circ
\beta(\sigma)^{-1}\circ \hat\kappa\circ \hat\mu_\sigma$, so $A^t$ satisfies
the property defining $X_\beta$. By the uniqueness property of $X_\beta$ we
have that $A^t\sim_k X_\beta$, and so $A_\beta\sim_k A $.
\end{proof}
\begin{remark}
The hypothesis  that $k$ is a number field has been used only in order to
guarantee the existence of splitting maps for $c_B$, by means of Tate's
theorem on the triviality of $H^2(G_k,\Fb^\times)$. Since Tate's theorem is
valid for any global or local field $k$, Theorem \ref{theorem: main theorem}
is valid for any global or local field $k$ as well.
\end{remark}

\section{Varieties over $k$ of $\GL_2$-type and $k$-varieties}\label{section: Varieties over $k$ of GL_2-type and k-varieties}
Let $k$ be a number field. In this section we characterize the absolutely
simple factors of the varieties over $k$ of $\GL_2$-type, in the case where
they do not have complex multiplication.
\begin{proposition}\label{proposition: GL2-type varieties are ribet-pyle}
Let $A/k$ be an abelian variety of $\GL_2$-type such that $A_\kb$ does not
have complex multiplication. Then $A$ is a Ribet-Pyle variety.
\end{proposition}
\begin{proof}
By \cite[Proposition 1.5]{shimura-class-fields} we can suppose that $A_\kb$
does not have any simple factor with CM. Let $A_\kb\sim B_1^{n_1}\times
\cdots \times B_r^{n_r}$ be the decomposition of $A_\kb$ into simple abelian
varieties up to isogeny. Since $E=\End_k^0(A)$ is a field it acts on each
factor $B_i^{n_i}$, and so it  acts on the homology with rational
coefficients $H_1((B_i^{n_i})_\C,\Q)$, which is a vector space of dimension
$2\dim B_i^{n_i}$ over $\Q$. Thus $2\dim B_i^{n_i}$ is divisible by
$[E:\Q]=\dim A$. But $\dim A\geq \dim B_i^{n_i}$, so either $[E:\Q]=\dim
B_i^{n_i}$ or $2[E:\Q]=\dim B_i^{n_i}$. The later is not possible, because
it would mean that $B_i^{n_i}$ has CM by $E$. Thus $\dim A=\dim B_i^{n_i}$
and $A_\kb$ has only one simple factor up to isogeny; say $A_\kb\sim B^n$.

Next, we see that $E$ is a maximal subfield of $\End_\kb^0(A)$. Let $C$ be
the centralizer of $E$ in $\End_\kb^0(A)$, and let $\varphi$ be an element
in $C$. A priori $\varphi (A_\kb)$ is isogenous to $B^r$ for some $r\leq n$.
Since $\varphi\in C$, the field $E$ acts on $\varphi(A_\kb)$; as before this
implies that $[E:\Q]$ divides $2\dim B^r$. But $[E:\Q]=\dim A=\dim B^n$,
therefore $r=n$ or $r=n/2$. Again $r=n/2$ is not possible, because then
$B^r$ would be a factor of $A_\kb$ with CM by $E$. Thus $r=n$ and $\varphi$
is invertible in $\End_\kb^0(A)$. This implies that $C$ is a field, and then
$E$ is a maximal subfield of $\End_\kb^0(B)$.

Finally, we see that $A_\kb$ is an abelian $k$-variety. For each $\sigma \in
G_k$ the map
\begin{equation}\label{equation: F-algebra endomorphism}
\begin{array}{ccc}
\End_\kb^0(A) & \lra & \End_\kb^0(A)\\
\varphi & \longmapsto & \acc\sigma \varphi
\end{array}
\end{equation}
is the identity when restricted to $E$. Since $E$ is a maximal subfield, it
contains the center $F$ of $\End_\kb^0(A)$, so \eqref{equation: F-algebra
endomorphism} is a $F$-algebra automorphism. By the Skolem-Noether Theorem
there exists an element $\mu_\sigma$ in $\End_\kb^0(A)^*$ such that
$\acc\sigma \varphi=\mu_\sigma^{-1}\circ \varphi\circ \mu_\sigma$, and we
see that $\mu_\sigma$ is a compatible isogeny in the sense of Definition
\ref{definition: abelian k-varieties}.
\end{proof}

\begin{definition}\label{def: building-k-block}
A \emph{building $k$-block} is an abelian $k$-variety $B/\kb$ such that
$\End_\kb^0(B)$ is a central division algebra over a field $F$, with Schur
index $t\leq 2$ and reduced degree $t[F:\Q]=\dim B$.
\end{definition}

\begin{theorem}
Let $k$ be a number field and let $A/k$ be an abelian variety of
$\GL_2$-type such that $A_\kb$ does not have CM. Then $A_\kb\sim B^n$ for
some building $k$-block $B$. Conversely, if $B$ is a building $k$-block then
there exists a variety $A/k$ of $\GL_2$-type such that $A_\kb\sim B^n$ for
some $n$.
\end{theorem}
\begin{proof}
By Propostion \ref{proposition: GL2-type varieties are ribet-pyle} $A$ is a
Ribet-Pyle variety, and by Proposition \ref{prop: Ribet-Pyle varieties are
isotypical} we have that $A_\kb\sim B^n$ for some $k$-variety $B$. Let
$\cB=\End_\kb^0(B)$, let $F$ be the center of $\cB$ and let $t$ be its Schur
index. Then $E=\End_k^0(A)$ is a maximal subfield of $\End_\kb^0(A)\simeq
\M_n(\cB)$, which has dimension $n^2t^2$ over $F$. Therefore $ [E:F]=nt$,
and multiplying both sides of this equality by $[F:\Q]$ we see that
$[E:\Q]=\dim A= nt[F:\Q]$. The equality $t[F:\Q]=\dim B$ follows. Since
$\cB$ is a division algebra of $\Q$-dimension $t^2[F:\Q]$ that acts on
$H_1(B_\C,\Q)$, which has $\Q$-dimension $2\dim B=2t[F:\Q]$, we see that
necessarily $t\leq 2$ and $B$ is a building $k$-block.

Conversely, let $B$ be a building $k$-block. In particular it is a
$k$-variety, and by Theorem \ref{theorem: Ribet-Pyle variety associated to a
k-variety} there exists a Ribet-Pyle variety $A/k$ such that $A_\kb\sim B^n$
for some $n$. The field $E=\End_k^0(A)$ is a maximal subfield of
$\End_\kb^0(A)\simeq \M_n(\cB)$, which means that $[E:F]=nt$. Multiplying
both sides of this equality by $[F:\Q]$ we see that $[E:\Q]=nt[F:\Q]=n\dim
B=\dim A$, and so $A$ is a variety of $\GL_2$-type.
\end{proof}
 In the case $k=\Q$ the center of the
endomorphism algebra of a building $k$-block is necessarily  totally real,
but for arbitrary number fields $k$ a priori it can be either totally real
or CM. That is why in  Definition \ref{def: building-k-block} the field $F$
is not required to be totally real. However, if $k$ admits a real embedding
then exactly the same argument of \cite[Theorem 1.2]{pyle} shows that $F$ is
necessarily  totally real. In addition, there are  some extra  restrictions
on the endomorphism algebra.
\begin{proposition}
Let  $k$ be a number field that admits a real embedding. Let $B$ be a
building $k$-block, let $\cB=\End_\kb^0(B)$ and let $F=Z(\cB)$. Then $F$ is
totally real and $\cB$ is either isomorphic to $F$ or to a totally
indefinite division quaternion algebra over $F$.
\end{proposition}
\begin{proof}
We view $k$ as a subfield of $\C$ by means of a real embedding
$k\hookrightarrow \R$. Let $A/k$ be a $\GL_2$-type variety such that
$A_\kb\sim B^n$. Let $E$ be the maximal subfield $\End_k^0(A)$ of
$\End_\kb^0(A)$, and identify $F$ with $Z(\End_\kb^0(A))$; under this
identification $F$ is contained in $E$. Let $t$ be the Schur index of $B$
and let $m=2\dim B/[\cB:\Q]$, for which we have that $mt=2$.

The division algebra $\cB$ belongs a priori to one of the four types of
algebras with a positive involution, according to Albert's classification
(see for instance \cite[Proposition 1]{shimura-on-analytic}). However, type
III is not possible; indeed by \cite[Proposition 15]{shimura-on-analytic}
the variety $B$ would then be isogenous to the square of a CM abelian
variety.

To see that type IV is also not possible, suppose that $F$ is a CM extension
of a totally real field $F_0$. Let $\Phi$ denote the complex representation
of $\cB$ on the space of differential forms $H^0(B_\C,\Omega^1)$. For every
real embedding $\nu$ of $F_0$ let $\chi_\nu,\overline\chi_\nu$ be the two
complex-conjugate irreducible representations of $\cB$ extending $\nu$. Let
$r_\nu$ and $s_\nu$ be the multiplicities of $\chi_\nu$ and
$\overline\chi_\nu$ in $\Phi$. For each $\nu$ we have that $r_\nu+s_\nu=2$;
moreover, the equality $r_\nu=s_\nu=1$ is not possible for all $\nu$ (cf.
\cite[Propositions 18 and 19]{shimura-on-analytic}). This implies that
$\mathrm{Tr}(\Phi)_{|F}=\sum r_\nu
{\chi_\nu}_{|F}+s_\nu{\overline\chi_\nu}_{|F}$ takes non-real values. On the
other hand, if we denote by $\Psi$ the complex representation of
$\End^0_\kb(A)$ on $H^0(A_\C,\Omega^1)$, then
$\mathrm{Tr}(\Psi)=n\mathrm{Tr}(\Phi)$. Since $A$ is defined over $k$ we can
take a basis of the differentials defined over $k$, and with respect to this
basis the elements of $E$ are represented by matrices with coefficients in
$k$. Since $F\subseteq E$, the trace of $\Psi$ restricted to $F$ takes
values in $k\subseteq \R$, giving a contradiction with the fact that
$\mathrm{Tr}(\Phi)_{|F}$ takes non-real values.
\end{proof}

\subsection*{Acknowledgements} I am grateful to my advisor,
Professor Jordi Quer, for his help and guidance throughout this
work. I would also like to thank Francesc Fité for carefully
reading  a previous version of this manuscript.


\end{document}